\documentclass[12pt]{amsart}
\usepackage{amssymb}
\usepackage{amsmath}
\newcommand\AND{\quad\text{and}\quad}
\newcommand\Dcal{\mathcal{D}}
\newcommand\diam{\text{\rm diam}}
\newcommand\Dir{\mathcal{E}}

\newcommand\ep{\varepsilon}

\newcommand\HD{\mathcal {HD}}
\newcommand\Lap{\mathcal{L}}

\newcommand\La{\Lambda}

\newcommand\msf{\mathsf m}

\newcommand\Prob{\mathsf{Pr}}

\newcommand\Q{\mathbb{Q}}
\newcommand\R{\mathbb R}

\newcommand\supp{\operatorname{\rm supp}}

\newcommand\uno{\mathbf{1}}
\newcommand\wh{\widehat}
\newcommand\wt{\widetilde}

\numberwithin{equation}{section}

\newtheoremstyle{mythm}
  {9pt}
  {9pt}
  {\itshape}
  {0pt}
  {\bfseries}
  {}
  { }
  {\thmnumber{(#2)}\thmname{ #1}\thmnote{ #3}}

\newtheoremstyle{mydef}
  {9pt}
  {9pt}
  {\normalfont}
  {0pt}
  {\bfseries}
  {}
  { }
  {\thmnumber{(#2)}\thmname{ #1}\thmnote{ #3}}

\theoremstyle{mythm}
\newtheorem{thm}[equation]{Theorem.}

\newtheorem{lem}[equation]{Lemma.}

\theoremstyle{mydef}
\newtheorem{dfn}[equation]{Definition.}

\newtheorem{rmk}[equation]{Remark.}

\begin{document}$\,$ \vspace{-1truecm}
\title{On the duality between jump processes on ultrametric spaces
and random walks on trees}
\author{\bf Wolfgang WOESS}
\address{\parbox{.8\linewidth}{Institut f\"ur Mathematische Strukturtheorie 
(Math C),\\ 
Technische Universit\"at Graz,\\
Steyrergasse 30, A-8010 Graz, Austria\\}}
\email{woess@TUGraz.at}
\date{July 12, 2012} 
\subjclass[2000] {05C05, 
31C05, 
60G50,	
60J50 
}
\keywords{Compact ultrametric space, tree, boundary, random walk, jump process,
Dirichlet form, Na\"\i m kernel}
\begin{abstract}
The purpose of these notes is to clarify the duality between a natural class
of jump processes on compact ultrametric spaces -- studied in current work of 
{\sc Bendikov, Girgor'yan and Pittet}~\cite{BGP1}, \cite{BGP2} -- and nearest neighbour walks on trees. 
Processes of this type have appeared in recent work of {\sc Kigami}~\cite{Ki}.
Every compact ultrametric space arises as the boundary of a locally finite
tree. The duality arises via the Dirichlet forms: one on the tree associated
with a random walk and the other on the boundary of the tree, which is given
in terms of the Na\"\i m kernel. Here, it is explained that up to a linear time
change by a unique
constant, there  is a one-to-one correspondence between the above processes
and Dirichlet regular random walks.

\smallskip

For definitive publication, an adapted version of these notes will be integrated into 
the paper \cite{BGP2}.

\end{abstract}


\maketitle

\markboth{{\sf W. Woess}}
{{\sf Processes on ultrametric spaces
and random walks on trees}}
\baselineskip 15pt

\section{Introduction}\label{sec:intro}

In recent work, {\sc Kigami}~\cite{Ki} starts with a transient
nearest neighbour random walk on a tree and constructs a naturally
associated jump process on the boundary of the tree, a Cantor set.
Using this approach, he undertakes a detailed analysis of the process
on the boundary. 

Arriving from a completely different viewpoint, {\sc Bendikov, 
Grigor'yan and Pittet}~\cite{BGP1} introduce a very natural class of processes on 
discrete, non-compact ultrametric spaces and provide a detailed analsysis of
them. In ongoing work by the same authors~\cite{BGP2}, this elegant approach 
is generalised to non-discrete
ultrametric spaces, compact as well a non-compact.
Those spaces are assumed to be complete and locally compact, and to possess
no isolated points. There is a natural way how such an ultrametric space
arises as 
the geometric boundary at infinity of a locally 
finite rooted tree 
where each vertex has at least two forward neighbours. 

The purpose of this note is to answer the obvious question (posed to
the author by A. Bendikov) how the approaches of Kigami and of Bendikov,
Grigor'yan and Pittet are related in the compact case (compactness is 
inherent in Kigami's work but not necessary for the approach of \cite{BGP1},
\cite{BGP2}): the relation is basically one-to-one.

``Basically'' means that we need to restrict to random walks on trees which are
\emph{Dirichlet regular,} that is, the Dirichlet problem at infinity
admits solution, or equivalently, the Green kernel of the random walk
vanishes at infinity. We comment on this condition, which 
appears to be natural in the present context, at the end of \S\ref{sec:duality}.

Given such a nearest neighbour random walk on a tree, its reversibility leads to
an interpretation of the tree as an infinite electric network. This comes along
with a natural Dirichlet form on the tree, and a natural approach is to use the
Dirichlet form on the boundary which reproduces the power (``energy'')
of harmonic functions on the tree via their boundary values. This form on the
boundary is computed with some effort in \cite{Ki}; it induces the jump 
process studied there. Now, that form on the boundary is an integral
with respect to the \emph{Na\"\i m kernel}, which goes back to
the work of {\sc Na\"\i m} \cite{Na} and {\sc Doob}~\cite{Do} in the setting
of abstract potential theory on spaces which are locally Euclidean.
Trees do not have the latter property, but the validity of the resulting
fomula for the power of harmonic functions is proved for general infinite 
electric networks ($\equiv$ reversible random walks) in a forthcoming
paper of {\sc Kaimanovich and Georgakopoulos}~\cite{GK}. A direct and rather 
simple proof for the case of trees is also given in the present note.

This paper is basically expository, based on a certain experience of the
author in handling nearest neighbour random walks on trees as in Chapter
9 of his book \cite{W-Markov}. Some results are new, resp. simplify or clarify
the approach of \cite{Ki}. Some emphasis is laid on introducing all the
background without rush and in a unified notation. 

In \S\ref{sec:ultra-tree},
it is explained how ultrametric spaces are related with trees and their
geometric boundaries. We introduce the notion of an \emph{ultrametric element}
on a tree, which induces an ultrametric on the boundary of the tree. Different
ultrametric elements induce the same topology on the boundary.

In \S\ref{sec:processes}, we present the two classes of stochastic processes
which are the subject of this paper. In \S\ref{sec:processes}.{\bf A},
we present the isotropic jump processes of \cite{BGP2} on an ultrametric
space, which is assumed right away to be the boundary of a tree.  
Every isotropic jump process relies on three input data: an ultrametric element
$\phi$, a probability measure $\mu$ on the boundary of the tree and a probability 
measure $\sigma$ on $\R^+$. The process is called \emph{standard,} if $\sigma$
is the inverse exponential distribution. With a suitable change of the 
ultrametric, 
but the same $\mu$, every process becomes a standard process.

In \S\ref{sec:processes}.{\bf B}, we outline the basic facts regarding
transient nearest neighbour random walks on trees. We present the Na\"\i m kernel
and state the theorem that relates the Dirichlet form on the space
of harmonic functions with finite power with the Dirichlet form on the boundary
that is computed via that kernel in terms of the boundary values of the 
involved harmonic functions.

In \S\ref{sec:duality}, we finally explain the relation between the 
jump processes of \cite{BGP2} and the processes of \cite{Ki} induced on the
boundary of a tree by a random walk on that tree.
In \S\ref{sec:duality}.{\bf A}, it is shown that every boundary process induced
by a random walk is an isotropic jump process in the sense of \cite{BGP2}.
In \S\ref{sec:duality}.{\bf B}, it is shown that up to a unique linear time change,
every isotropic jump process on the boundary of a tree arises from a uniquely
determined random walk as the process of \cite{Ki}.
\S\ref{sec:duality}.{\bf C} contains a discussion on some of the assumptions
and a few clarifications (in random walk terminology)
regarding some of the statements in \cite{Ki}.
 
The Appendix, \S\ref{sec:appendix}, contains a proof of the Doob-Na\"\i m formula
for nearest neighbour random walks on trees.

We conclude this introduction by a brief overview on previous work, inculding also
non-compact ultrametric spaces, to which the methods of {\sc Bendikov,
Grigor'yan and Pittet}~\cite{BGP1}, \cite{BGP2} apply equally well.

The classical example of a non-compact ultrametric space is the field 
$\Q_p$ of $p$-adic numbers. A family a Laplace-type operators on $\Q_p\,$ 
was introduced and studied in the context of $p$-adic analysis by
{\sc Vladimirov} and collaborators; see the book by
{\sc Vladimirov, Volovich and Zelenov}~\cite{VVZ} and its references. 
The Vladimirov Laplacian
is recovered in an elegant way by the approach of \cite{BGP2}.

{\sc Fig\`a-Talamanca}~\cite{Fi} and 
{\sc Del Muto and Fig\`a-Talamanca}~\cite{DeFi1}, \cite{DeFi2} -- see also 
 and 
{\sc Baldi, Casadio-Tarabusi and Fig\`a-Talamanca}~\cite{BCF} --
have constructed
diffusion processes on homogeneous ultrametric spaces, including the 
$p$-adic number field, via a harmonic analysis
approach. The construction on $\Q_p$ starts with a discrete process on a level set
of the homoegenous tree, and then  a rescaling procedure leads to a process
on the ``lower'' boundary of the tree, which is $\Q_p$. 

{\sc Kochubei}~\cite{Ko} (+ references in that book)
has undertaken a careful analysis of the Vladimirov Laplacian, closely
related with the isotropic jump processes of \cite{BGP2}.


{\sc Albeverio and Karwowski}~\cite{AK1}, \cite{AK2} construct continuous-time
processes on the $p$-adics, and more generally on the non-compact
ultrametric space which is the lower boundary of a tree, based on a suitable
Chapman-Kolmogorov equation.

More or less at the same time as Kigami, {\sc Pearson and Bellisard}~\cite{PB} 
have introduced a ``Laplace-Beltrami'' operator on Cantor sets via trees and
so-called spectral triples and studied spectrum as well as the ``ddiffusion''
process associated with the operator.

When one compares several of these approaches with the simple and clear 
construction of
\cite{BGP1}, \cite{BGP2}, one may be astonished about the enormous notational 
and technical efforts undertaken to construct the respective processes of 
those references.
 
\section{Compact ultrametric spaces and trees}\label{sec:ultra-tree}

{\bf A. Ultrametric spaces}

A metric space $(X,d)$ is called \emph{ultrametric} if instead of the triangle
inequality it satisfies the stronger \emph{ultrametric inquality}
$$
d(\xi,\zeta) \le \max \{ d(\xi,\eta), d(\eta,\zeta)\} \quad \text{for all }\;
\xi, \eta, \zeta \in X\,.
$$
In these notes, we assume that our ultrametric space is \emph{compact}
and has \emph{no isolated points.} We recall that it is
totally disconnected, that every ball 
$$
B_d(\xi, r) = B(\xi, r) = \{ \eta \in X : d(\eta, \xi) \le r \}
$$
is open and compact, and that $B(\xi,r) = B(\eta,r)$ for every 
$\eta \in B(\xi,r)$. In particular, for any fixed $r > 0$, the collection
of all balls  with radius $r$ is an open cover of $X$. The set
$$
\La_d(\xi)  = \{d(\eta, \xi) : \eta \in X\,,\; \eta \ne \xi \}
$$
is countable, bounded by compactness, and has $0$ as its only
accumulation point. (It \emph{is} an accumulation point by the ``no 
isolated points'' assumption.) The set 
$\La_d = \bigcup_{\xi \in X} \La_d(\xi)$ has the same properties.

\bigskip

{\bf B. Infinite trees}

A tree is a connected graph $T$ without circles (closed paths of length 
$\ge 3$). We tacitly identify $T$ with its vertex set, which is assumed to be
infinite. We write $x \sim y$
if $x, y \in T$ are neighbours. For any pair of vertices
$x, y \in T$, there is a unique shortest path, called
\emph{geodesic segment}   
$$
\pi(x,y) = [x=x_0\,, x_1\,, \dots, x_k =y]
$$
such that $x_{i-1} \sim x_i$ and all $x_i$ are disctinct. 
If $x=y$ then this is the \emph{empty} or \emph{trivial} path. The
number $k$ is the  \emph{length} of the path (the graph distance 
between $x$ and $y$).

In $T$ we choose and fix a \emph{root vertex} $o$. We write $|x|$
for the length of $\pi(o,x)$. The choice of the root induces a partial order on
$T$, where $x \le y$ when $x \in \pi(o,y)$. Every $x \in T \setminus \{o\}$
has a unique \emph{predecessor} $x^-$, which is the unique neighbour of $x$ 
on $\pi(o,x)$. Thus, the set of all (unoriented) edges of $T$ is
$$
E(T) = \{[x^-,x] : x \in T\,,\; x \ne o \}\,.
$$
For $x \in T$, the number
$$
\deg^+(x) = |\{ y \in T: y^-=x\}|
$$ is the \emph{forward degree} of $x$. 
We assume that $T$ is \emph{locally finite,} that is, $\deg^+(x) < \infty$, and that it
has no \emph{dead ends,} that is $\deg^+(x) \ge 1$ for every $x \in T$.
Below, we shall even assume that $\deg^+(x) \ge 2$. 

A \emph{(geodesic) ray} in $T$ is a one-sided infinite path
$\pi = [x_0\,, x_1\,, x_2\,, \dots]$  such that $x_{n-1} \sim x_n$ and all 
$x_n$ are disctinct. Two rays are \emph{equivalent} if their symmetric
difference (as sets of vertices) is finite. An \emph{end} of $T$ is an 
equivalent class of rays. The set of all ends of $T$ is denoted $\partial T$.
This is the \emph{boundary} at infinity of the tree. For any $x \in T$ and 
$\xi \in \partial T$, there is a unique ray $\pi(x,\xi)$ which is a
representative of the end (equivalence class) $\xi$ and starts at $x$. 
We write 
$$
\wh T = T \cup \partial T. 
$$
For $x \in T$, the \emph{branch of $T$ rooted at $x$} is the
subtree $T_x$ that we identify with its set of vertices
$$
T_x = \{ y \in T : x \le y\}\,,
$$
so that $T_o=T$. We write $\partial T_x$ for the set of all ends of
$T$ which have a representative path contained in $T_x$, and
$\wh T_x = T_x \cup \partial T_x\,$.

For $w, z \in \wh T$, we define their \emph{confluent} 
$w \wedge z = w \wedge_o z$ with respect to the 
root $o$ by the relation 
$$
\pi(o,w \wedge z) = \pi(o,w) \cap \pi(o,z)\,.
$$
It is the last common element on the geodesics $\pi(o,w)$ and $\pi(o,w)$,
a vertex of $T$ unless $w=z \in \partial T$.

One of the best-known ways to define an ultrametric on $\wh T$ is
\begin{equation}\label{eq:tree-metr} 
d_e(z,w) = \begin{cases} 0\,,&\text{if}\; z=w\,,\\
e^{-|z \wedge w|} \,,&\text{if}\; z\ne w\,.
\end{cases}
\end{equation}
Then $\wh T$ is compact, and $T$ is open and dense. We are mostly interested
in the compact ultrametric space $\partial T$. In the metric $d_e$ of 
\eqref{eq:tree-metr}, each ball with centre $\xi \in \partial T$ is of the form
$\partial T_x$ for some $x \in \pi(o,\xi)$. Indeed, if we set
$\pi^+(o,x) = \{ x \in \pi(o,\xi) : \deg^+(x) \ge 2 \}$
then 
$$
\partial T_x = B_{d_e}(\xi,e^{-|x|}) \quad \text{for every}\; x \in \pi^+(o,x)\,,\AND
\La_{d_e}(x) = \{ e^{-|x|} : x \in \pi^+(o,\xi)\}\,.
$$
Also, $\xi$ is isolated if and only if $\pi_+(o,\xi)$ is finite.
Here, we shall mainly be interested in the opposite situation, 
when $\deg^+(x) \ge 2$ for  every $x$.
 
 \newpage

{\bf C. The tree associated with an ultrametric space}

We now start with a compact ultrametric space $(X,d)$ that does not possess
isolated points, and construct a tree $T$ as follows:
The vertex set of $T$ is
$$
\bigl\{ B_d(\xi, r) : \xi \in X\,,\; r > 0\bigr\}
$$
Here, we may assume (if we wish) that $r \in \La_d(\xi)$.

Given $\xi \in X$ and $r \in \La_d(\xi)$, take any 
$\eta \in B_d(\xi,r) = B_d(\eta,r)$. Since $\eta$ is not isolated, there must 
be $r_{\eta} \in \La_d(\eta)$ with $r_{\eta} < r$ and such that
$\La_d(\eta) \cap \bigl(r_{\eta}\,,\,r\bigr) = \emptyset$.
We call the ball $y = B_d(\eta,r_{\eta})$ a successor of $x = B_d(\xi,r)$. 
As vertices of our tree, this defines neighbourhood, where $x = y^-$.
The root vertex $o$ is $X$, the ball with maximal radius.
By compactness, each $x$ has only finitely many successors, and since
there are no isolated points in $X$, every vertex has at least 2 successors.

This defines the tree structure. For any $\xi \in X$, the
collection of all balls $B(\xi,r)$, $r \in \La_d(\xi)$, ordered decreasingly,
forms the set of vertices of a ray in $T$ that starts at $o$. 
Via a straightforward exercise, the mapping that associates to $\xi$ the end
of $T$ represented by that ray is a homeomorphism from $X$ onto $\partial T$. 
Thus, we can identify $X$ and $\partial T$ as ultrametric spaces.

In this identification, if originally a vertex $x$ was interpreted as
a ball $B(\xi,r)$, $r \in \La_d(\xi)$, then the set $\partial T_x$ of
ends of the branch $T_x$ just coincides with the ball $B(\xi,r)$.
That is, we are identifying each vertex $x$ of $T$ with the 
set $\partial T_x$.
 
If we start with an arbitrary locally finite tree and take its space of
ends as the ultrametric space $X$, then the above construction does not
recover vertices with forward degree $1$, so that in general we do not get 
back the tree we started with. However, via the above construction,
the correspondence between compact ultrametric spaces without isolated points
and locally finite rooted trees with forward degrees $\ge 2$ is bijective. 

From now on, we can abandon the notation $X$ for our ultrametric space.

\smallskip

\emph{We consider $X$ as the boundary 
$\partial T$ of a locally finite, rooted tree with forward degrees $\ge 2$.}

\smallskip

 At the end, we shall comment on 
how one can handle the presence of vertices with forward degree $1$.

\medskip

There are many ways to equip $\partial T$ with an ultrametric that has
the same topology and the same compact-open balls $\partial T_x\,$, $x \in T$,
possibly with different radii than in the standard metric \eqref{eq:tree-metr}.

\begin{dfn}\label{def:element} Let $T$ be a locally finite, rooted tree $T$ 
with $\deg^+(x) \ge 2$ for all $x$. An \emph{ultrametric element}
is a function $\phi: T \to (0\,,\infty)$ with
$$
\begin{aligned} 
&\text{(i)}\quad \phi(x^-) > \phi(x)\quad\text{for every}\; 
x \in T \setminus \{o\}\,,\\
&\text{(ii)}\quad \lim \phi(x_n) = 0 \quad\text{along every geodesic ray}\;
\pi = [x_0\,, x_1\,,x_2\,,\dots]\,.
\end{aligned}
$$
It induces the ultrametric $d_{\phi}$ on $\partial T$ given by
$$
d_{\phi}(\xi,\eta) = \begin{cases} 0\,,&\text{if}\; \xi=\eta\,,\\
\phi(\xi \wedge \eta) \,,&\text{if}\; \xi\ne \eta\,.
\end{cases}
$$
\end{dfn} 
The balls in this ultrametric are again the sets 
$$
\partial T_x = B_{d_{\phi}}\bigl(\xi, \phi(x)\bigr)\,,\quad 
\xi \in  \partial T_x\,.
$$
Note that condition (ii) in the definition is needed for having that
each end of $T$ is non-isolated in the metric $d_{\phi}\,$. The standard metric
$d_e$ is of course induced by $\phi(x) = e^{-|x|}$.

\begin{lem}\label{lem:dphi} For a tree as in Definition
\ref{def:element}, every ultrametric on $\partial T$ whose closed balls
are the sets $\partial T_x\,$, $x \in T$, is induced by an ultrametric 
element on $T$. 
\end{lem}

\begin{proof} Given an ultrametric $d$ as stated, we set 
$\phi(x) = \diam(\partial T_x)$, the diameter with respect to the metric $d$.
Since $\deg^+(x^-) \ge 2$ for any $x \in T \setminus \{o\}$,
the ball $\partial T_{x^-}$ is the disjoint union of at least two balls
$\partial T_y$ with $y^-=x^-$. Therefore 
we must have  $\diam(\partial T_x) < \diam(\partial T_{x^-})$, and property
(i) holds. Since no end is isolated, $\phi$ satisfies (ii). It is now
straighforward that $d_{\phi} = d$.
\end{proof}

In view of this correspondence, in the sequel we shall replace the subscript
$d$ referring to the metric $d = d_{\phi}$ by the subscript $\phi$ referring
to the ultrametric element.
We note that 
\begin{equation}\label{eq:range}
\diam_{\phi}(\partial T) = \phi(o)\,,\quad
\La_{\phi}(\xi) = \{ \phi(x) : x \in \pi(o,\xi) \} \AND
\La_{\phi} = \{ \phi(x) : x \in T \}.
\end{equation}
We also note here that for any $\xi \in \partial T$ and $x \in \pi(o,\xi)$,
\begin{equation}\label{eq:balls} 
B_{\phi}(\xi,r) =  \begin{cases} \partial T_x\, &\text{for}\; \phi(x) \le r <
\phi(x^-)\,,\; \text{if}\; x \ne o\\
\partial T&\text{for}\; r \ge \phi(o)\,,\; \text{if}\; x = o\,.
\end{cases}\end{equation}

\section{Stochastic processes on ultrametric spaces and on trees}
\label{sec:processes}

{\bf A. Isotropic jump processes on a compact ultrametric space}

We start with a compact ultrametric space $(X,d)$ without isolated points.
We realise $X$ as $\partial T$, where $T$ is a rooted, locally finite
tree with forward degrees $\ge 2$. By Lemma \ref {lem:dphi}, we can represent
$d = d_{\phi}$ for the associated ultrametric element $\phi$ on $T$.
 
We choose and fix a Borel probability measure $\mu$ on $X = \partial T$
with $\supp(\mu) = \partial T$, and choose another probability measure 
$\sigma$ on the non-negative
reals $\R_+$ whose distribution function 
$F_{\sigma}(r) = \sigma\bigl( [0\,,\,r)\bigr)$
has the following properties:\footnote{In \cite{BGP1} and \cite{BGP2}, 
the measure $\sigma$ is denoted
$c$, while $\mu$ is denoted $m$. We avoid this, because it is in slight
conflict with part of our notation, which is imported from the book
\cite{W-Markov}.}
\begin{equation}\label{eq:sigma}
F_{\sigma}(0+) = 0\,,\;  F_{\sigma}\bigl(\phi(o)\bigr) < 1\,, \quad 
\text{and it is strictly increasing on each}\; \La_{\phi}(\xi),\;\xi \in \partial T.
\end{equation}
The measure $\mu$ is fixed throughout, while we shall ``play'' with
$\sigma$ as well as with the ultrametric $d$ (resp. its ultrametric element
$\phi$).

{\sc Bendikov, Grigor'yan and Pittet}~\cite{BGP2} (see also \cite{BGP1})
have introduced a class of stochastic
processes on $X$ in the following way: on $L^2(\partial T,\mu)$, define
\begin{equation}\label{eq:average}
P_rf(\xi) = \frac{1}{\mu(B_d(\xi,r))} \int_{B_d(\xi,r)} f\,d\mu
\end{equation}
and the transition operator
\begin{equation}\label{eq:isotropic1}
Pf(\xi) = \int_{\R^+} P_rf(\xi)\,d\sigma(r)\,.
\end{equation}
This is a self-adjoint, bounded Markov operator. In view of the identity
$P_rP_s = P_{\max\{r,s\}}\,$, we can construct 
the powers $P^t$, $t > 0$, of $P$ 
\begin{equation}\label{eq:isotropic2}
P^tf(\xi) = \int_{\R^+} P_rf(\xi)\,d\sigma^t(r)\,,
\end{equation}
where $\sigma^t$ is the probability distribution on $\R^+$ with
$F_{\sigma^t}(r) = F_{\sigma}(r)^t\,$. 
Then $(P^t)_{t > 0}$ is a continuous Markov semigroup,
which is Fellerian. It gives rise to a Hunt process $(X_t)_{t \ge 0}\,$ on
our ultrametric space. We call it the \emph{$(\phi,\mu,\sigma)$-process} on
$\partial T$.
Its properties and an elegant and refined analysis are the core of 
\cite{BGP2}
and will not be repeated here in detail. In view of its definition,
we can call the process isotropic: informally speaking,
its transition kernel is equidistributed
on each sphere. The conditions \eqref{eq:sigma} on the measure $\sigma$
are natural in order to guarantee that the process is irreducible, i.e.,
that it can reach every open set from any starting point with positive
probability. 
In more detail, let $\xi \in \partial T$ and 
$\pi(0,\xi) = [o=x_0\,,x_1\,,x_2\,,\dots]$. Then, using \eqref{eq:balls}, 
\begin{equation}\label{eq:Pt}
\begin{gathered}
P^tf(\xi)= \sum_{n=0}^{\infty} c_n^t \, P_{\phi(x_n)}f(\xi)\,,\\\text{where}\quad 
c_0^t = \sigma^t\bigl([\phi(x_0)\,,\,\infty)\bigr) \AND 
c_n^t = \sigma^t\bigl([\phi(x_{n-1})\,,\,\phi(x_n)\bigr) \; \text{ for}\; n \ge 1\,.
\end{gathered}
\end{equation}
In particular, for arbitrary $y \in T$,
$$
\Prob[X_t \in \partial T_y \mid X_0=\xi] 
= \sum_{n=0}^{\infty} c_n^t \, 
\frac{\mu(\partial T_{x_n} \cap \partial T_y)}{\mu(\partial T_{x_n})}.
$$
The assumption that $F_{\sigma}\bigl(\phi(o)\bigr) < 1$
serves to guarantee that the process can exit with positive probability from each 
$\partial T_x\,$, where $x\sim o$.

We see that we have some freedom in the choice of the measure $\sigma$:
any two measures whose distribution functions coincide on the value set 
$\La_{\phi}$ of $d_{\phi}$ give rise to the same process. 

\begin{dfn}\label{def:standard} The \emph{standard process} associated with $\mu$ and
$d$ (resp., with the associated ultrametric element $\phi$) is the one where
$\sigma=\sigma_*$ is the ``inverse exponential distribution'' 
whose distribution function is
$$
\sigma_*\bigl([0\,,\,r)\bigr) = e^{-1/r}\,,\quad r > 0\,.
$$
\end{dfn}

\begin{lem}\label{lem:standard} Every process defined via \eqref{eq:average}
and \eqref{eq:isotropic1} is the standard process with respect to $\mu$ and
a suitably modified ultrametric element $\phi_*$.
\end{lem}

\begin{proof}
Let the original ultrametric be $d = d_{\phi}\,$. Thus $\diam(X) =
\diam(\partial T) = \phi(o)$, and by assumption,
$\sigma\bigl((0\,,\,\phi(o)\bigr) < 1$.
We define 
\begin{equation}\label{eq:phi*}
\phi_*(x) = -1\big/\log F_{\sigma}\bigl( \phi(x) \bigr)\,.
\end{equation}
By \eqref{eq:sigma}, this is an ultrametric element on $T$, and
for any $x\in T$
$$
F_{\sigma}\bigl( \phi(x) \bigr) = F_{\sigma_*}\bigl( \phi_*(x) \bigr) 
$$
Looking at \eqref{eq:Pt}, we see that $P_t$ is  induced by $\mu$, $\phi_*$ 
and $\sigma_*$ equally as by $\mu$, $\phi$ and $\sigma$, as required.
\end{proof}
  
One of the key features of \cite{BGP2} is an explicit formula for the infinitesimal
generator and the Dirichlet form of the Markov semigroup $(P_t)$.
Here, we shall have no need to go into details regarding the theory
of Dirichlet forms; see the standard reference {\sc Fukushima, Oshima and
Takeda}~\cite{FOT}. What will be important for
us is that the Dirichlet form determines the process, resp. the semigroup.

We start with the standard process one associated with $\mu$ and $\phi$.
In \cite{BGP2}, the following kernel $J=J_{\phi,\mu}$ is introduced.
\begin{equation}\label{eq:J}
\begin{aligned}
J(\xi,\eta) &= \int_{0}^{1/d_{\phi}(\xi,\eta)}  
   \frac{dt}{\mu\bigl( B_{\phi}(\xi,1/t)\bigr)} \\
&= \frac{1}{\phi(o)} + \int_{1/\phi(o)}^{1/\phi(\xi\wedge\eta)}
   \frac{dt}{\mu\bigl( B_{\phi}(\xi,1/t)\bigr)}\,,\quad \xi, \eta \in \partial T\,.
\end{aligned}
\end{equation}
We have $J(\xi,\eta) = J(\eta,\xi)$.
By \cite{BGP2}, 
the Dirichlet form 
$\Dir$ of the standard $(\phi,\mu)$-process is given, for functions 
$\varphi, \psi \in C(\partial T)$, by
\begin{equation}\label{eq:Dir}
\Dir_{\partial T}(\varphi, \psi) = \frac12 \int_{\partial T} \int_{\partial T} 
\Bigl(u(\xi)-u(\eta)\Bigr) \Bigl(v(\xi)-v(\eta)\Bigr) \, J(\xi,\eta) 
\,d\mu(\xi)\,d\mu(\eta).
\end{equation}
If we start with $\phi$, $\mu$ and a general measure $\sigma$ as in
\eqref{eq:sigma}, then we first represent the associated process as
the standard $(\phi_*,\mu)$-process, where $\phi_*$ is
as in\eqref{eq:phi*}. Then the Dirichlet form of the $(\phi,\mu,\sigma)$-process
is the one of \eqref{eq:Dir} with $J = J_{\phi_*,\,\mu}$ in the place of 
$J_{\phi,\mu}\,$. 

The infinitesimal generator $\Lap$ is given by
$\Lap u(\xi) = \int_{\partial T} \Bigl(u(\xi)-u(\eta)\Bigr)\,J(\xi,\eta) 
\,d\mu(\xi)\,,
$
but this will not be used here.
  
\bigskip

{\bf B. Nearest neighbour random walks on trees}

A good part of the material outlined in this subsection is taken
from the author's (little read) book \cite{W-Markov}. An older, recommended 
reference is the seminal paper of {\sc Cartier}~\cite{Ca}.

A \emph{nearest neighbour random walk} on the locally finite, infinite tree $T$ is 
induced by its stochastic transition matrix $P= \bigl(p(x,y)\bigr)_{x,y \in T}\,$
with the property that $p(x,y) > 0$ if and only if $x \sim y$. 
The resulting discrete-time Markov chain (random walk) is written $(Z_n)_{n\ge 0}\,$.
Its $n$-step transition probabilities
$p^{(n)}(x,y) = \Prob_x[Z_n=y]$, $x, y \in T$, are the elements of the
$n^{\text{th}}$ power of the matrix $P$. The notation $\Prob_x$ refers
to the probability measure on the trajectory space that governs the
random walk starting at $x$. We assume that the random walk is
\emph{transient,} i.e., with probability $1$ it visits any finite set only 
finitely often. Thus, $0 < G(x,y) < \infty$ for all $x,y \in X$,
where
$$
G(x,y) = \sum_{n=0}^{\infty} p^{(n)}(x,y)
$$ 
is the \emph{Green kernel} of the random walk. In addition, we shall also 
make crucial use of the quantities
$$
F(x,y) = \Prob_x[ Z_n =y \; \text{for some} \; n \ge 0]
\AND U(x,x) = \Prob_x[ Z_n =x \; \text{for some} \; n \ge 1]
\,.
$$
We shall need several identities relating them and start with a few of them,
valid for all $x,y \in T$.
\begin{align}
G(x,y) &= F(x,y) G(y,y)\label{eq:FG}\\
G(x,x) &=\frac{1}{1-U(x,x)}\label{eq:GU}\\
U(x,x) &=\sum_{y} p(x,y) F(y,x)\label{eq:UF}\\ 
F(x,y) &=F(x,z)F(z,y)\quad\text{whenever }\; z \in \pi(x,y)\label{eq:FF}
\end{align} 
The first three hold for arbitrary denumerable Markov chains, while
\eqref{eq:UF} is specific for trees (resp., a bit more generally, when $z$ is
a ``cut point'' between $x$ and $y$).
The identities show that those quantities can be determined just by all the
$F(x,y)$, where $x \sim y$. More identities, as to be found in 
\cite[Chapter 9]{W-Markov}, will be displayed and used later on.
By transience, the random walk $Z_n$ must converge to a random end, a simple and
well-known fact. See e.g. \cite{Ca} or \cite[Theorem 9.18]{W-Markov}.

\begin{lem}\label{lem:conv} There is a $\partial T$-valued random variable
$Z_{\infty}$ such that for every starting point $x \in T$,
$$
\Prob_x[Z_n \to Z_{\infty} \;\text{in the topology of}\;\wh T] =1.
$$
\end{lem}

In brief, the argument is as follows: by transience, random walk trajectories must
accumulate at $\partial T$ almost surely. If such a trajectory had two distinct
accumulation points, say $\xi$ and $\eta$, then by the nearest neighbour property,
the trajectory would visit the vertex $\xi\wedge\eta$ infinitely often, which
can occur only with probability $0$.

We can consider the family of \emph{limit distributions} $\nu_x\,$, $x \in T$,
where for any Borel set $B \in \partial T$,
$$
\nu_x(B) = \Prob_x[Z_{\infty} \in B]\,.
$$
The sets $\partial T_y\,$, $y \in T$, form a semiring that generates the
Borel $\sigma$-algebra of $\partial T$. Thus, each $\nu_x$ is determined by
the values of those sets. There is an explicit formula, compare with \cite{Ca} 
or \cite[Proposition 9.23]{W-Markov}. For $y \ne o$, 
\begin{equation}\label{eq:nu}
\nu_x(\partial T_y) = \begin{cases}
F(x,y) \dfrac{1-F(y,y^-)}{1- F(y^-,y)F(y,y^-)}\,,
&\text{if}\; x \in \{y\} \cup (T \setminus T_y)\,,\\[12pt]
1-F(x,y) \dfrac{F(y,y^-)-F(y^-,y)F(y,y^-)}{1- F(y^-,y)F(y,y^-)}\,,  
&\text{if}\; x \in T_y\,.
\end{cases}
\end{equation}
A \emph{harmonic function} is a function $h:T \to \R$ with $Ph = h$,
where
$$
Ph(x)=\sum_y p(x,y)h(y)\,.
$$
For any Borel set $B \subset \partial T$, the function $x \mapsto \nu_x(B)$
is a bounded harmonic function. One deduces that all $\nu_x$ are comparable:
$p^{(k)}(x,y) \, \nu_x \le \nu_y\,$, where $k$ is the length of $\pi(x,y)$.
Thus, for any $\varphi \in L^1(\partial T, \nu_o)$, the function $h_\varphi$ 
defined by
$$
h_{\varphi}(x) = \int_{\partial T} \varphi\,d\nu_x
$$
is finite and harmonic on $T$. It is often called the \emph{Poisson transform}
of $\varphi$. 

We next define a measure $\msf$ on $T$ via its atoms: 
$\msf(o)=1$, and for $x \in T \setminus \{o\}$ with 
$\pi(o,x) =[o=x_0\,, x_1\,, \dots, x_k=x]$,
\begin{equation}\label{eq:treerev}
\msf(x) = \frac{p(x_0,x_1)p(x_1,x_2) \cdots p(x_{k-1},x_k)}
{p(x_1,x_0)p(x_2,x_1)\cdots p(x_k,x_{k-1})}\,.
\end{equation}
Then for all $x,y \in T$,
\begin{equation}\label{eq:reversible}
\msf(x) p(x,y) = \msf(y) p(y,x)\,, \quad \text{and consequently}\quad
\msf(x) G(x,y) = \msf(y) G(y,x)\,; 
\end{equation}
the random walk is \emph{reversible}. This would allow us to use the 
\emph{electrical network} interpretation of $(T,P,\msf)$, for which there
are various references: see e.g. {\sc Yamasaki}~\cite{Ya},
{\sc Soardi}~\cite{So}, or -- with notation
as used here -- \cite[Chapter 4]{W-Markov}. We do not go into its details here;
each edge $e= [x^-,x] \in E(T)$ is thought of as an electic conductor with
\emph{conductance}
$$
a(x^-,x) = \msf(x) p(x,x^-).
$$ 
We get the Dirichlet form $\Dir_T = \Dir_{T,P}$ for functions $f, g: T \to \R$,
defined by 
\begin{equation}\label{eq:Dir-T}
\Dir_T(f,g) = \sum_{[x^-,x]\in E(T)} \bigl(f(x)-f(x^-)\bigr) 
\bigl(g(x)-g(x^-)\bigr) \, a(x^-,x)\,.
\end{equation}
It is well defined for $f, g$ in the space
\begin{equation}\label{eq:D(T)}
\Dcal(T)=\Dcal(T,P) = \{ f:  T \to\R \mid \Dir_T(f,f) < \infty\}.
\end{equation}
We are interested in the subspace 
$$
\HD(T)=\HD(T,P) = \{ h \in \Dcal(T,P): Ph =h < \infty\}
$$
of hamonic functions with \emph{finite power.} The terminology comes
from the interpretation of such a function as the potential of an
electric flow (or current), and then $\Dir_T(h,h)$ is the power of that
flow.\footnote{In the mathematical literature, mostly the expression ``energy'' is
used for $\Dir_T(h,h)$, but the author's modest understanding of Physics induces 
him to think that ``power'' is more appropriate.} 

Every function in $\HD(T,P)$ is the Poisson transform of
some function $\varphi \in L^1(\partial T,\nu_o)$. This is valid not only for trees, but
for general finite range reversible Markov chains,
and follows from the following facts. (1) Every function in
$\HD$ is the difference of two non-negative functions in $\HD$. (2) Every
non-negative function in $\HD$ can be approximated, monotonically from below, by
a sequence of non-negative bounded functions in $\HD$. (3) Every bounded harmonic
function (not necessarily with finite power) is the Poisson transform of a 
bounded function on the boundary. In the general setting, the latter is the 
(active part of) the Martin boundary, with $\nu_x$ being the limit distribution
of the Markov chain, starting from $x$, on that boundary.  (1) and (2) are
contained in \cite{Ya} and \cite{So}, while (3) is part of general Martin 
boundary theory, see e.g. \cite[Theorem 7.61]{W-Markov}.

Thus, we can introduce a form $\Dir_{\HD}$ on $\partial T$ by setting
\begin{equation}\label{eq:DirHD} 
\begin{gathered}
\Dcal(\partial T,P) = \{ \varphi \in L^1(\partial T,\nu_o) : 
\Dir_T(h_{\varphi},h_{\varphi}) < \infty\}\,,\\
\Dir_{\HD}(\varphi, \psi) = \Dir_T(h_{\varphi},h_{\psi}) \quad \text{for}\; 
\varphi, \psi \in \Dcal(\partial T,P).
\end{gathered}
\end{equation}
{\sc Kigami}~\cite{Ki} elaborates an expression for 
this form by considerable ``bare hands'' effort,
shows its regularity properties and then studies the process on $\partial T$
induced by this Dirichlet form. We call this the \emph{K-process} (``K'' for
``Kigami'') associated with the random walk. 

Now, there is a simple expression for
$\Dir_{\HD}\,.$ We define the \emph{Na\"\i m kernel} on 
$\partial T \times \partial T$ by
\begin{equation}\label{eq:Naim}
\Theta_o(\xi,\eta) = \begin{cases}
\dfrac{\msf(o)}{G(o,o)F(o,\xi\wedge\eta) F(\xi\wedge\eta,o)}\,,
&\text{if}\; \xi \ne \eta\,,\\
+\infty\,,&\text{if}\; \xi =\eta\,.
\end{cases}
\end{equation}
In our case, $\msf(o)=1$, but we might want to change the base point,
or normalise the measure $\msf$ in a different way.

\begin{thm}\label{thm:doob-naim}
For any transient nearest neighbour random walk on the tree $T$
with root $o$, and all functions $\varphi$, $\psi$ in $\Dcal(\partial T,P)$,
\begin{equation}\label{eq:doob-naim}
\Dir_{\HD}(\varphi,\psi) = 
\frac12 \int_{\partial T} \int_{\partial T} 
\Bigl(\varphi(\xi)-\varphi(\eta)\Bigr) \Bigl(\psi(\xi)-\psi(\eta)\Bigr)\Theta_o(\xi,\eta) 
\,d\nu_o(\xi)\,d\nu_o(\eta)\,.
\end{equation}
\end{thm}

There is a general definition of the Na\"\i m kernel \cite{Na} that involves the
Martin boundary, which in the present case is $\partial T$. 
A proof of Theorem \eqref{eq:doob-naim} is given in \cite{Do} in a setting of abstract
potential theory on Green spaces, which are locally Euclidean. The definition of
\cite{Na} refers to the same type of setting. Now,
infinite networks, even when seen as metric graphs, are not locally Euclidean.
In this sense, so far the definition of the kernel and a proof of  
\eqref{eq:doob-naim} for transient, reversible random walks
have not been available in the literature. Indeed, such a proof in the network
setting is highly desirable for the large audience who might not be so well
familiar with abstract potential theory in the style of the 1960s.
In a forthcoming paper, {\sc Georgakopoulos and Kaimonvich}~\cite{GK} provide those
``missing links''. Nevertheless, in the appendix, we shall give a direct and 
simple proof of Theorem \ref{thm:doob-naim} for the specific case of trees.\footnote{As a matter of fact, based on a
mysterious remark in {\sc Kaimanovich}~\cite{Ka}, the author had initiated to
elaborate this proof in the late 1990s, but in the phase of transition from
Italy to Austria, the hand-written notes were lost and the project was not 
pursued until now.}

\section{The duality between random walks on $T$ and jump processes
on $\partial T$}\label{sec:duality}

When looking at the isotropic jump processes of \cite{BGP2}  on $\partial T$,
as described in \S~\ref{sec:processes}.{\bf A}, and at the ones introduced in 
\cite{Ki}, outlined
at the end of \S~\ref{sec:processes}.{\bf B}, it is natural to ask the following two
questions.

\medskip

(I) Given a transient random walk on $T$, does the K-process on $\partial T$
induced by the form $\Dir_{\HD}$ of \eqref{eq:DirHD} arise as one of the isotropic 
jump processes described in \S~\ref{sec:processes}.{\bf A} with respect to the measure 
$\mu = \nu_o$ on $\partial T$, some ultrametric element $\phi$ on $T$
and a suitable measure $\sigma$ on $\R^+$ ?

\medskip

(II) Conversely, given $\mu$, $\phi$ and $\sigma$, is there a random walk on $T$
with limit distribution $\nu_o=\mu$ such that the jump process induced by
$\mu$, $\phi$ and $\sigma$ is the K-process with Dirichlet form $\Dir_{\HD}$ ?

\medskip

Before answering both questions in the next two subsections, we need to specify the
assumptions more precisely. When starting with $(\phi,\mu,\sigma)$, it is natural
to assume that $\mu$ is supported by the whole of $\partial T$, as we have
assumed in \S~\ref{sec:processes}.{\bf A}. Indeed, otherwise the state space of the 
$(\phi,\mu,\sigma)$-process would reduce to the ultrametric sub-space $\supp(\mu)$
of $\partial T$, and as long as $\supp(\mu)$ is required to have no isolated points,
we can invoke \S~\ref{sec:ultra-tree}.{\bf C} to build up an associated tree with forward degrees $\ge 2$
directly from $\supp(\mu)$. 

Thus, on the side of the random walk, we also want that $\supp(\nu_o) = \partial T$.
This is equivalent with the requirement that $\nu_o(\partial T_x) > 0$ for every 
$x$. By \eqref{eq:nu} this is in turn equivalent with 
\begin{equation}\label{eq:F<1} F(x,x^-) <1 \quad
\text{for every}\; x \in T \setminus \{o\}.
\end{equation} 
Indeed, we shall see that we need a bit more, namely that 
\begin{equation}\label{eq:Green-0}
\lim_{x \to \infty} G(x,o) = 0\,.
\end{equation}
(I.e., for every $\ep > 0$ there is a finite set $F \subset T$ such that
$G(x,o) < \ep$ for all $x \in T \setminus F$.) This condition is necessary and
sufficient for solvability of the \emph{Dirichlet problem:} for any $\varphi \in C(\partial X)$,
its Poisson transform $h_{\varphi}$ provides the unique continuous extension
of $\varphi$ to $\wh T$ which is harmonic in $T$. See e.g. 
\cite[Corollary 9.44]{W-Markov}.

We shall restrict attention to random walks with properties \eqref{eq:F<1}
and \eqref{eq:Green-0} on a rooted tree with forward degrees $\ge 2$.

\bigskip

{\bf A. Answer to question (I)}
 
We start with a random walk that fulfills the above requirements.
We know from Lemma \ref{lem:standard}
that each $(\mu,\phi,\sigma)$-process arises as the standard process with respect
to a modification $\phi_*$ of $\phi$. Thus, we can eliminate $\sigma$
from our considerations by just looking for an ultrametric element $\phi$
such that the K-process is the standard process on $\partial T$
associated with $(\nu_0\,,\phi)$. 

Since the processes are determined by the Dirichlet forms, we infer from
\eqref{eq:Dir} and \eqref{eq:doob-naim} that we are looking for $\phi$
such that $J_{\phi,\nu_o}(\xi,\eta) = \Theta(\xi,\eta)$ for all 
$\xi, \eta \in \partial T$ with $\xi \ne \eta$. This becomes
\begin{equation}\label{eq:tosolve1}
\frac{1}{\phi(o)} + \int_{1/\phi(0)}^{1/\phi(\xi\wedge\eta)} 
\frac{dt}{\nu_o(B_{\phi}(\xi,1/t))} = 
\frac{1}{G(o,o)F(o,\xi\wedge\eta) F(\xi\wedge\eta,o)}.
\end{equation}
First of all, since $\deg^+(o) \ge 2$, there are
$\xi,\eta\in \partial T$ such that $\xi \wedge \eta=o$.
We insert these two boundary points in \eqref{eq:tosolve1}. Since $F(o,o)=1$, 
we see that we must have
$$
\phi(o) = G(o,o)\,.
$$
Now take $x \in T\setminus \{o\}$. Since forward degrees are $\ge 2$, there are
$\xi,\eta,\eta' \in \partial T$ such that $\xi \wedge \eta=x$ and
$\xi \wedge \eta'=x^-$. We write \eqref{eq:tosolve1} first for $(\xi,\eta')$ and 
then for $(\xi,\eta)$ and then take the difference, leading to
the equation 
\begin{equation}\label{eq:tosolve2}
\int_{1/\phi(x^-)}^{1/\phi(x)} \frac{dt}{\nu_o(B_{\phi}(\xi,1/t))} = 
\frac{1}{G(o,o)F(o,x) F(x,o)} - \frac{1}{G(o,o)F(o,x^-) F(x^-,o)}.
\end{equation}
By \eqref{eq:balls}, within the range of the last integral we must have
$B_{\phi}(\xi,1/t) = \partial T_x\,$, whence that integral reduces to
$$
\left(\frac{1}{\phi(x)} - \frac{1}{\phi(x^-)}\right)\frac{1}{\nu_o(\partial T_x)} 
$$
We multiply equation \eqref{eq:tosolve2} by $\nu_o(\partial T_x)$ and 
simplify the resulting right hand side 
$$
\left(\frac{1}{G(o,o)F(o,x) F(x,o)} - \frac{1}{G(o,o)F(o,x^-) F(x^-,o)}\right)
\nu_o(\partial T_x)
$$
by use the identities \eqref{eq:FG} -- \eqref{eq:FF} and the first of
the two formulas of \eqref{eq:nu} (for $\nu_o$). We obtain that the ultrametric
element that we are looking for should satisfy
\begin{equation}\label{eq:solution1}
\frac{1}{\phi(x)} - \frac{1}{\phi(x^-)} = \frac{1}{G(x,o)} - \frac{1}{G(x^-,o)}
\quad \text{for every}\; x \in T \setminus \{o\}\,.
\end{equation}
This determines $1/\phi(x)$ recursively, and 
we get
$$
\phi(x) = G(x,o)\,.
$$  
Since $G(x,o) = F(x,x^-)G(x^-o)$ by \eqref{eq:FG} and \eqref{eq:FF},
the assumptions \eqref{eq:F<1} and \eqref{eq:Green-0} yield that  $\phi$
is indeed an ultrametric element. Tracing back the
last computations, we find that with this choice of $\phi$,
we have indeed that $J(\xi,\eta) = \Theta(\xi,\eta)$ for all 
$\xi, \eta \in \partial T$
with $\xi \ne \eta$. We have proved the following.

\begin{thm}\label{thm:I} Let $T$ be a locally finite, rooted tree with forward degrees
$\ge 2$, and consider a transient nearest neighbour random walk on $T$ that
satisfies \eqref{eq:FG} and \eqref{eq:FF}. 

Then the K-process on $\partial T$ induced by the Dirichlet form 
\eqref{eq:DirHD} $\equiv$ \eqref{eq:doob-naim} coincides with the
standard process associated with ultrametric element $\phi = G(\cdot,o)$ and
the limit distribution $\nu_o$ of the
random walk.
\end{thm}

\begin{rmk}\label{rmk:sigma}
Given the random walk on $T$ and the associated $K$-process on $\partial T$,
we might want to realise it at the $(\nu_o, \phi, \sigma)$-process
for an ultrametric element $\phi$ different from $G(\cdot,o)$.
This means that we have to look for a suitable measure $\sigma$ on $\R^+$.
Let us now write $\phi_*(x) = G(x,o)$.
In view of Lemma \ref{lem:standard} and its proof, we are looking
for $\sigma$ such that for our given generic $\phi$,
$$
F_{\sigma}\bigl(\phi(x)\bigr) = e^{-1/G(x,o)}.
$$
For this it is necessary that $\phi(x)=\phi(y)$ whenever $G(x,o)=G(y,o)$: we need
$\phi$ to be constant on equipotential sets.
In that case, $F_{\sigma}(r)$ is determined by the above equation for $r$ in
the value set $\La_{\phi}$ of the ultrametric $d_{\phi}\,$. 
We can ``interpolate'' that function in an arbitrary
way (monotone increasing, left continuous) and get a feasible measure $\sigma$.   
\end{rmk}

\bigskip

{\bf B. Answer to question (II)}
   
Answering question (II) means that we start with $\phi$ and $\mu$
and then look for a random walk with limit distribution $\nu_o=\mu$
such that the standard $(\phi,\mu)$-process is the K-process associated with
the random walk. We know from Theorem \ref{thm:I} that in this case,
we should have $\phi(x) = G(x,o)$, whence in particular, $\phi(o) > 1$. 
Thus we cannot expect that every $\phi$ is suitable. The most natural
choice is to replace $\phi$ by $C \cdot\phi$ for some constant $C > 0$.
For the standard processes associated with $\phi$ and $C\cdot\phi$, respectively,
this just gives rise of a linear time change: if the old process
is $(X_t)$, then the new one is $(X_{t/C})$.

\begin{thm}\label{thm:II} Let $T$ be a locally finite, rooted tree with forward degrees
$\ge 2$, and consider an ultrametric element $\phi$ on $T$ and a fully supported
probability measure $\mu$ on $\partial T$. 

Then there are a unique constant
$C > 0$ and a unique transient nearest neighbour random walk on $T$ that
satisfies \eqref{eq:FG} and \eqref{eq:FF} with the following properties:
$\mu = \nu_o\,$ the limit distribution of the random walk, and
the associated K-process on $\partial T$ coincides with the
standard process induced by the ultrametric element $C\cdot\phi$ and
the given measure~$\mu$.  
\end{thm}

For the proof, we shall need three more formulas. The first two are taken
from \cite[Lemma 9.35]{W-Markov}, while the third is immediate from
\eqref{eq:nu} and \eqref{eq:FF}
\begin{align}
G(x,x)\,p(x,y) &= \frac{F(x,y)}{1- F(x,y)F(y,x)}\,
\;\;\text{if}\;y \sim x\,, \AND\label{eq:formula1}\\
G(x,x) &= 1 + \sum_{y: y \sim x} \frac{F(x,y)F(y,x)}{1- F(x,y)F(y,x)}
\label{eq:formula2}\\
F(x^-,x) &= \frac{\nu_o(\partial T_x)/F(o,x^-)}
{1- F(x,x^-) + F(x,x^-)\,\,\nu_o(\partial T_x)/F(o,x^-)}.\label{eq:formula3}
\end{align}

\begin{proof}{\bf Proof of Theorem \ref{thm:II}}

We proceed as follows: we start with $\phi$ and $\mu$ and
replace $\phi$ by a new ultrametric element $C\cdot \phi$, with $C$ to be determined,
and $\mu$ being the candidate for the limit distribution of the random walk that
we are looking for.

Using the various formulas
at our disposal, we first construct in the only possible way the 
quantities $F(x,y)\,$, $x,y,\in T$, in particular when $x \sim y$.
In turn, they lead to the Green kernel $G(x,y)$. So far, these will
be only ``would-be'' quantities whose feasability will have to be verified.
Until that verification, we shall denote them $\wt F(x,y)$ and $\wt G(x,y)$.
Via \eqref{eq:formula1}, they will lead to definitions of transition
probabilities $p(x,y)$. Stochasticity of the resulting transition matrix $P$
will also have to be verified. 

Only then, we will use a potential theoretic argument to show that
$\wt G(x,y)$ really is the Green kernel associated with $P$, so that the ``?''
can be removed.  

\smallskip

First of all, in view of Theorem \ref{thm:I}, we must have
$$
C \cdot \phi(x) = \wt G(x,o)\,, 
$$
whence by \eqref{eq:FG} and \eqref{eq:FF} 
\begin{equation}\label{eq:solveF1}
\wt F(x,x^-) = \phi(x)/\phi(x^-) \quad \text{for}\; x \in T \setminus \{o\}\,,
\end{equation}
and more generally
$$
\wt F(y,x) = \phi(y)/\phi(x)\quad \text{when} \; x \le y\,.
$$  
We note immediately that $0 < \wt F(y,x) < 1$ when $x < y$, and that $\wt F(x,x)=1$.

Next, we use \eqref{eq:formula3} to construct recursively $\wt F(x^-,x)$ and
$\wt F(o,x)$. We start with $\wt F(o,o) = 1$.
If $x \ne o$ and $\wt F(o,x^-)$ is already given, with $\mu(\partial T_{x^-})\le 
\wt F(o,x^-) \le 1$ (the lower bound is required by \eqref{eq:nu}),  
then we have to set 
\begin{equation}\label{eq:solveF2}
\wt F(x^-,x) = \frac{\mu(\partial T_x)/\wt F(o,x^-)}
{1- \wt F(x,x^-) + \wt F(x,x^-)\,\,\mu(\partial T_x)/\wt F(o,x^-)}
\end{equation}
and $\wt F(o,x) = \wt F(o,x^-) \wt F(x^-,x)\,$.
Since $\wt F(o,x^-) \ge \mu(\partial T_{x^-}) \ge \mu(\partial T_x)$,
we see that $0 < \wt F(x^-,x) \le 1$.  We set -- as imposed by \eqref{eq:FF} --
$\wt F(o,x) = \wt F(o,x^-)\wt F(x^-,x)$. Formula \eqref{eq:solveF2} (re-)transforms into
\begin{equation}\label{eq:re-mu}
\mu(\partial T_x) = \wt F(o,x^-)\wt F(x^-,x)
\frac{1-\wt F(x,x^-)}{1-\wt F(x,x^-)\wt F(x^-,x)} \le \wt F (o,x) \le 1\,, 
\end{equation}
as needed for our recursive construction.
Now we have all $\wt F(x,y)$, initially for $x \sim y$, and consequently for
all $x,y$ by taking products along geodesic paths.

We now can compute the constant $C$: \eqref{eq:formula2}, combined with
\eqref{eq:solveF1} and \eqref{eq:re-mu} for $x \sim o$ forces 
$$
\begin{aligned}
C \phi(o) &= \wt G(o,o) 
=  1 + \sum_{x: x \sim o} \frac{\wt F(o,x)\wt F(x,o)}{1- \wt F(o,x)\wt F(x,o)}
&=1 + \sum_{x: x \sim o} \frac{\wt F(x,o)}{1- \wt F(x,o)}\mu(\partial T_x)\\
&= 1 + \sum_{x: x \sim o} \frac{\phi(x)/\phi(o)}{1- \phi(x)/\phi(o)}
\mu(\partial T_x) 
\end{aligned}
$$
Therefore 
\begin{equation}\label{eq:C}
C = \frac{1}{\phi(o)} + \sum_{x: x \sim o} \frac{\phi(x)/\phi(o)}{\phi(o)- \phi(x)}
\mu(\partial T_x)\,. 
\end{equation} 
We now construct $\wt G(x,x)$ via \eqref{eq:formula2}:
\begin{equation}\label{eq:Gxx}
\wt G(x,x) = 1 + \sum_{y: y \sim x} \frac{\wt F(x,y)\wt F(y,x)}{1- \wt F(x,y)\wt F(y,x)}.
\end{equation}
For $x = o$, we know that this is compatible with our choice of $C$. 
At last, our only choice for the Green kernel is
$$
\wt G(x,y) = \wt F(x,y) \wt G(y,y)\,, \quad x,y \in T\,.
$$
Now we finally arrive at the only way how to define the transition probabilites,
via \eqref{eq:formula1}: 
\begin{equation}\label{eq:pxy}
p(x,y) =\frac1{\wt G(x,x)} \frac{\wt F(x,y)}{1- \wt F(x,y)\wt F(y,x)}.
\end{equation}
\emph{Claim 1.}\quad $P$ is stochastic.

\smallskip

\noindent \emph{Proof of Claim 1.}
Combining \eqref{eq:pxy} with  \eqref{eq:Gxx}, we deduce that we have to verify
that for every $x \in T$,
\begin{equation}\label{eq:verify}
\sum_{y:y\sim x} \frac{\wt F(x,y)(1-\wt F(y,x))}{1- \wt F(x,y)\wt F(y,x)} = 1\,.
\end{equation}
If $x=o$, then by \eqref{eq:re-mu} this is just 
$\sum_{y:y\sim o} \mu(\partial T_y) = 1$.
If $x \ne o$ then, again by \eqref{eq:re-mu}, the left hand side of 
\eqref{eq:verify} is
$$
\begin{aligned}
\sum_{y:y^- =x} &\frac{\wt F(y^-,y)(1-\wt F(y,y^-))}{1- \wt F(y^-,y)\wt F(y,y^-)}
 + \frac{\wt F(x,x^-)(1-\wt F(x^-,x))}{1- \wt F(x,x^-)\wt F(x^-,x)}
 \\
&= \sum_{y:y^- = x} \frac{\mu(\partial T_y)}{\wt F(o,x)}
+ 1 - \frac{1-\wt F(x^-,x)}{1- \wt F(x,x^-)\wt F(x^-,x)}
= 1.
\end{aligned}
$$
This proves Claim 1.

\smallskip
\noindent\emph{Claim 2.} For any $x_0\in T$, the function 
$\tilde g_{x_0}(x) = \wt G(x,x_0)$ satisfies $P\tilde g_{x_0} = \tilde g_{x_0} -\uno_{x_0}\,$.

\smallskip

\noindent \emph{Proof of Claim 2.}
First, we combine \eqref{eq:Gxx} with \eqref{eq:pxy} to combine
$$
P\tilde g_{x_0}(x_0) = \sum_{y:y \sim x_0} p(x_0,y)\wt F(y,x_0) \wt G(x_0,x_0)
= \sum_{y:y \sim x_0} \frac{\wt F(x_0,y)\wt F(y,x_0)}{1- \wt F(x_0,y)\wt F(y,x_0)}
= \tilde g_{x_0}(x_0)-1\,,
$$
and Claim 2 is true for $x=x_0\,$. Second, for $x \ne x_0$, let $u$ be the 
neighbour of $x$ on $\pi(x,x_0)$. Then
$$
\begin{aligned}
&P\tilde g_{x_0}(x) 
= \sum_{y:y \sim x, y \ne u} p(x,y)\wt F(y,x) \wt G(x,x_0)
+ p(x,u)\wt G(u,x_0)\\
&\quad= \underbrace{\sum_{y:y \sim x} 
\frac{\wt F(x,y)\wt F(y,x)}{1- \wt F(x,y)\wt F(y,x)}}_{\displaystyle
\wt G(x,x)-1}\frac{\wt G(x,x_0)}{\wt G(x,x)} 
- p(x,u) \wt F(u,x)\wt G(x,x_0) + p(x,u)\wt G(u,x_0)\\
&\quad= G(x,x_0)\left(1 - \frac{1}{\wt G(x,x)} -p(x,u) \wt F(u,x)
+ p(x,u)\frac{1}{\wt F(x,u)}\right) = \tilde g_{x_0}(x)\,,
\end{aligned}
$$
since $p(x,u)/\wt F(x,u) - p(x,u)\wt F(u,x) = 1/\wt G(x,x)$ by \eqref{eq:pxy}. This
completes the proof of Claim 2.

\smallskip

Now we can conclude: the function $\tilde g_{x_0}$ is non-constant, positive and 
superharmonic.
Therefore the random walk with transition matrix $P$ given by \eqref{eq:pxy}
is transient and does posses a Green function $G(x,y)$. Furthermore, by the
Riesz decomposition theorem, we have
$$
\tilde g_{x_0} = Gf +h\,,
$$
where $h$ is a non-negative harmonic function and the charge $f$
of the potential $Gf(x) = \sum_y G(x,y)f(y)$ is 
$f = \tilde g_{x_0} - P\tilde g_{x_0}= \uno_{x_0}\,$.  That is,
$$
\wt G(x,x_0) = G(x,x_0) + h(x) \quad \text{for all}\;x
$$
Now let $\xi \in \partial T$ and $y = x_0 \wedge \xi$. 
If $x \in T_y$ then by our construction
$$
\wt G(x,x_0)= \wt G(x,o) \frac{\wt G(y,x_0)}{\phi(y)} \phi(x) \to 0 \quad 
\text{as}\; x \to \xi\,.
$$ 
Therefore $\wt G(\cdot, x_0)$ vanishes at infinity, and the same must hold
for $h$. By the maximum principle, $h \equiv 0$. 

We conclude
that $\wt G(x,y) = G(x,y)$ for all $x,y\in T$. But then, by our construction,
also $\wt F(x,y) =F(x,y)$, the ``first hitting'' kernel associated with $P$. 
Comparing $\eqref{eq:re-mu}$ with $\eqref{eq:nu}$, we see that
$\mu = \nu_o\,$. This completes the proof.
\end{proof}
 
\bigskip

{\bf C. Some remarks}

\begin{rmk}\label{rmk:non-compact} In the present notes, we have restricted
attention to \emph{compact} ultrametric spaces for two reasons. First, Kigami's 
approach \cite{Ki}
starting with a random walk on a rooted, locally finite tree $T$ 
is restricted to that situation, because $\partial T$ is compact. 

On the other hand, the approach of \cite{BGP2} is not restricted to
compact spaces.  In case of a non-compact, locally compact tree, one constructs
the tree in the same way (the vertex set corresponds to the collection of all 
closed balls), but then the tree will have ``its root at infinity'', i.e.,
the ultrametric space becomes $\partial^* T = \partial T \setminus \{\varpi\}$,
where $\varpi$ is a fixed reference end of $T$. In this situation, the predecessor
$x^-$ of a vertex $x$ is the neighbour of $x$ on $\pi(x,\varpi)$. In the definition
\ref{def:element} of an ultrametric element, we then need besides monotonicity 
that $\phi$ tends to $\infty$ along $\pi(x,\varpi)$, while it has to tend to 
$0$ along any  geodesic going to $\partial^* T$. In this setting, the
construction of a $(\phi,\mu,\sigma)$-process remains as in \eqref{eq:isotropic1}
and \eqref{eq:isotropic1}, but $\mu$ may have infinite mass: a Radon measure
supported on the whole of $\partial^* T$.

In the non-compact case, however -- and this is the second reason -- 
at the moment we do not see
an elegant and concise interpretation (analogous
to the $K$-process)  of a $(\phi,\mu,\sigma)$-process in terms of a random walk. 
\end{rmk}

\begin{rmk}\label{rmk:mixed} Here, as in \cite{BGP2}, we have always assumed
that the ultrametric space has no isolated points, which for the tree means
that $\deg^+ \ge 2$. Theme of \cite{BGP1} is the opposite situation, where
all points are isolated, i.e., the space is discrete.

From the point of view of the present notes, the mixed situation works 
equally well. If we start with a locally compact ultrametric space having
both isolated and non-isolated points, we can construct the tree in the same
way. The vertex set is the collection of all closed balls. The isolated
points will then become \emph{terminal vertices} of the tree, which have no 
neighbour besides the predecessor. All interior (non terminal) vertices will 
have forward degree $\ge 2$. 

In the compact case, the boundary $\partial T$ of that tree should consist of 
the terminal vertices together with the space of ends.
In the non-compact case, we will again have a reference end $\varpi$ 
as in Remark \ref{rmk:non-compact}. 

The definition of an ultrametric element
remains the same, but we only need to define it on interior vertices.
In this general setting, the construction of $(\phi,\mu,\sigma)$ processes
remains unchanged.

When the ultrametric space is compact, even in presence of isolated points,
the duality with random walks on the associated tree remains as explained
here. The random walk should then be such that the terminal vertices are
absorbing, and that the Green kernel tends to $0$ at infinity. The Doob-Na\"\i m
formula extends readily to that setting.
\end{rmk}

\begin{rmk}\label{rmk:rwstart}
Let us now consider the general (compact) situation when we start with a
transient random walk on a locally finite tree, rooted $T$. 

The limit distribution
$\nu_o$ will in general not be supported by the whole of $\partial T$. 
The $K$-process can of course still be constructed, see \cite{Ki}, but will
evolve naturally on $\supp(\nu_o)$ only. Thus, we can consider our ultrametric
space to be just $\supp(\nu_o)$, as mentioned further above. The tree associated
with this ultrametric space will in general not be the tree we started with,
nor its \emph{transient skeleton} as defined in \cite[(9.27)]{W-Markov} 
(the subtree induced by $o$ and all $x \in T\setminus\{o\}$ with $F(x,x^-) < 1$).

The reasons are twofold. First, the construction of the tree associated with
$\supp(\nu_o)$ will never give back vertices with forward degree $1$. Second,
some end contained in $\supp(\nu_o)$ may be isolated within that set, while
not being isolated in $\partial T$. But then this element will become
a terminal vertex in the tree associated with the ultrametric (sub)space
$\supp(\nu_o)$. This occurs precisely when the transient skeleton has
isolated ends. 

Thus, one should work with a modified ``reduced'' tree plus random walk
in order to maintain the duality between random walks and isotropic jump
processes.
\end{rmk}

\begin{rmk}\label{rmk:regularity} 
Given a transient random walk on the tree $T$, 
{\sc Kigami}~\cite{Ki} recovers an \emph{intrinsic metric} of the $K$-process
on $\partial T$ in terms of what is called an ultrametric element in the
present paper. This is of course $\phi(x) = G(x,o)$, denoted $D_x$ in
\cite{Ki}, where it is shown that for $\nu_o$-almost every $\xi \in \partial T$,
$D_x \to$ along the geodesic ray $\pi(o,\xi)$. This has the following 
Potential theoretic interpretation. 

A point $\xi \in \partial T$ is called \emph{regular for the Dirichlet problem,}
if for every $\varphi \in C(\partial T)$, its Poisson transform $h_{\varphi}$
satisfies $\lim_{x \to \xi} h_{\varphi}(x) = \varphi(\xi)$. 
It is known from {\sc Cartwright, Soardi and Woess}~\cite[Remark 2]{CSW} that
$\xi$ is regular if and only if  $\lim_{x \to \xi} G(x,o)=0$ (as long as $T$ has
at least $2$ ends), see also \cite[Theorem 9.43]{W-Markov}. By 
that theorem 9.4, the set of regular points has $\nu_o$-measure 1.
That is, the Green kernel vanishes at $\nu_o$-almost every boundary point.
\end{rmk}

\begin{rmk}\label{rmk:vondra}
In the proof of Theorem \ref{thm:II}, we have reconstructed random
walk transition probabilities from $C\cdot\phi(x) = G(x,o)$ and
$\mu = \nu_o\,$. 

A similar (a bit simpler) question was addressed by {\sc Vondra\v cek}~\cite{Vo}:
how to reconstruct the transition probabilities from \emph{all} limit distributions
$\nu_x\,$, $x \in T$, on the boundary. This, as well as out method, basically
come from \eqref{eq:formula2}, which can be traced back to {\sc Cartier}~\cite{Ca}.
\end{rmk}

\section{Appendix: a proof of the Doob-Na\"\i m formula on trees}
\label{sec:appendix}

We start with the following observation.

\begin{lem}\label{lem:invariant}
The measure $\Theta_o(\xi,\eta) \,d\nu_o(\xi)\,d\nu_o(\eta)$ 
on $\partial T \times \partial T$ is invariant with respect to changing
the base point (root) $o$.
\end{lem}

\begin{proof} We want to replace the base point $o$ with some other $x \in T$.
We may assume that $x \sim o$. Indeed, then we may step by step replace
the current base point by one of its neighbours to obtain the result for
arbitrary $x$.

Recall that the confluent that appears in the definition \eqref{eq:Naim} of 
$\Theta_o$ depends on
the root $o$, while for $\Theta_x$ it becomes the one with respect to $x$.
It is a well-known fact that 
$$
\frac{d\nu_x}{d\nu_o}(\xi) = K(x,\xi) = 
\frac{G(x,x\wedge_o \xi)}{G(o,x\wedge_o \xi)}\,,
$$
the \emph{Martin kernel.}
Thus, we have to show that for all $\xi,\eta \in \partial T$ ($\xi \ne \eta$)
$$
\dfrac{\msf(o)}{G(o,o)F(o,\xi\wedge_o\eta) F(\xi\wedge_o\eta,o)} = 
\dfrac{\msf(x)K(x,\xi) K(x,\eta)}{G(x,x)F(x,\xi\wedge_x\eta) F(\xi\wedge_x\eta,x)} 
\,.
$$

\medskip

\emph{Case 1.} $\xi, \eta \in \partial T_x$. Then
$\xi\wedge_o\eta = \xi\wedge_x\eta =:y \in T_x$,
and $x\wedge_o \xi =x\wedge_o \eta =x$.
Thus, using \eqref{eq:FG}, \eqref{eq:FF} and the fact that 
$\msf(x)/G(o,x) = \msf(o)/G(x,o)$ by \eqref{eq:reversible}
$$
\begin{aligned}
\frac{\msf(x)K(x,\xi) K(x,\eta)}{G(x,x)F(x,\xi\wedge_x\eta) F(\xi\wedge_x\eta,x)}  
&= \frac{\msf(x)}{G(x,x)F(x,y) F(y,x)} \left(\frac{G(x,x)}{G(o,x)}\right)^2\\ 
&= \frac{\msf(o)G(x,x)}{F(x,y) F(y,x)G(o,x)G(x,o)}\\
&= \frac{\msf(o)}{F(x,y) F(y,x)F(o,x)F(x,o)G(o,o)}\\
&= \frac{\msf(o)}{F(o,y) F(y,o) G(o,o)}\,, 
\end{aligned}
$$
as required.

\smallskip

\emph{Case 2.} $\xi, \eta \in \partial T \setminus \partial T_x\,.$. Then
$\xi\wedge_o\eta = \xi\wedge_x\eta =:z \in T \setminus T_x$,
and $x\wedge_o \xi =x\wedge_o \eta =o$. 

\smallskip

\emph{Case 3.} $\xi \in \partial T_x\,$, 
$\eta \in \partial T \setminus \partial T_x$. Then
$\xi\wedge_o\eta = o$, $\xi\wedge_x\eta =x$,
$x\wedge_o \xi =x$ and $x \wedge_o \eta =o$. 

\smallskip

\emph{Case 4.} $\xi \in \partial T \setminus  \partial T_x\,$, 
$\eta \in \partial T_x\,$. This is like case 3, exchanging the roles of 
$\xi$ and $\eta$.

\smallskip

In all those cases, the computation is done very similarly to case 1,
a straightforward exercise.
\end{proof}

For proving Theorem \ref{thm:doob-naim}, we need a few more facts related 
with the network setting; compare e.g. with \cite[\S 4.D]{W-Markov}.

The space $\Dcal(T)$ of \eqref{eq:D(T)} is a Hilbert space when equipped
with the inner product
$$
(f,g) = \Dir_T(f,g) + f(o)g(o)\,.
$$
The subspace $\Dcal_0(T)$ is defined as the closure of the space of finitely
supported functions in $\Dcal(T)$. It is a proper subspace if and only if the 
random walk is transient, and then the function $G_y(x) = G(x,y)$ is 
in $\Dcal_0(T)$ for any $y \in T$ \cite{Ya}, \cite{So}. We need the formula 
\begin{equation}\label{eq:Green}
\Dir_T(f,G_y) = \msf(y) f(y) \quad \text{for every}\; f \in \Dcal_0(T)\,.
\end{equation}
Given a branch $T_z$ of $T$ ($z \in T \setminus \{ o\}$),
we can consider it as a subnetwork equipped with the same conductances
$a(x,y)$ for $[x,y] \in E(T_z)$. The associated measure on $T_z$ is 
$$
\msf_{T_z}(x) = \sum_{y \in T_z: y \sim x} a(x,y) = 
\begin{cases} \msf(x)&\text{if}\; x \in T_z \setminus \{z\}\,,\\
\msf(z)-a(z,z^-)&\text{if}\; x =z\,. 
\end{cases}
$$
The resulting random walk on $T_z$ has transition probabilities
$$
p_{T_z}(x,y) = \frac{a(x,y)}{\msf_{T_z}(x)} = 
\begin{cases} p(x,y)&\text{if}\; x \in T_z \setminus \{z\}\,,\;y \sim x\,,\\[9pt]
\dfrac{p(z,y)}{1-p(z,z^-)}&\text{if}\; x =z\,,\;y \sim x\,. 
\end{cases}
$$
We have $F_{T_z}(x,x^-) = F(x,x^-)$ and thus also 
$F_{T_z}(x,z) = F(x,z)$ for every $x \in T_z \setminus \{z\}$, because before its
first visit to $z$,  the random walk on $T_z$ obeys the same transition 
probabilities as the original  random walk on $T$. It is then easy to see
\cite[p. 241]{W-Markov} that the random walk on $T_z$ is transient if
and only if for the original random walk, $F(z,z^-) <1$, which in turn holds
if and only $\nu_o(\partial T_z) > 0$. (In other parts of this paper, this
is always assumed, but for the proof of Theorem \ref{thm:doob-naim}, we 
just consider the random walk on the whole of $T$ to be transient.)     
Conversely, if $F(z,z^-)=1$ then $F(x,z)=1$ for all $x \in T_z\,$.

Below, we shall need the following formula for the 
limit distributions.

\begin{lem}\label{lem:nu} For $x \in T \setminus \{o\}$, 
$$
\nu_x(\partial T_x) =  1 - p(x,x^-)\bigl(G(x,x) - G(x^-,x)\bigr)\,.
$$
\end{lem}

\begin{proof}
By \cite[Lemma 9.35]{W-Markov},
$$
G(x,x)p(x,x^-) = \frac{F(x,x^-)}{1-F(x,x^-)F(x^-,x)}
$$
Thus, 
$$
p(x,x^-)\bigl(G(x,x) - G(x^-,x)\bigr) = \bigl(1-F(x^-,x)\bigr) G(x,x)p(x,x^-)
= 1 -\nu_x(\partial T_x)
$$ 
by a short computation using \eqref{eq:nu}
\end{proof}

\begin{proof}[\bf Proof of Theorem \ref{thm:doob-naim}]
We first prove formula \eqref{eq:doob-naim} for the case when
$\varphi= \uno_{\partial T_w}$ and $\psi=\uno_{\partial T_z}$ for two proper branches
$T_w$ and $T_z$ of $T$. They are either disjoint, or one of them contains
the other. 

\smallskip

\emph{Case 1. } $T_z \subset T_w\,.\quad$ (The case $T_w \subset T_z\,$ is 
analogous by symmetry.)

This means that $z \in T_w\,.$
For $\xi, \eta \in \partial T$ we have 
$\bigl(\varphi(\xi)-\varphi(\eta)\bigr) \bigl(\psi(\xi)-\psi(\eta)\bigr) =1$
if $\xi \in \partial T_z$ and $\eta \in \partial T \setminus \partial T_w$
or conversely, and $=0$ otherwise. By Lemma \ref{lem:invariant}, we may choose
$w$ as the base point.
Thus, the right hand side of \eqref{eq:doob-naim} is
$$
\int_{\partial T \setminus \partial T_w} \int_{\partial T_z} 
\Theta_w(\xi,\eta) \,d\nu_w(\xi)\,d\nu_w(\eta) = 
\dfrac{\msf(w)}{G(w,w)}\,\nu_w(\partial T \setminus \partial T_w)\nu_w(\partial T_z)
\,,
$$
since $\xi\wedge_w\eta =w$ and $F(w,w)=1$.

Let us now turn to the left hand side of \eqref{eq:doob-naim}. 
The Poisson transforms of $\varphi$ and $\psi$ are
$$
h_{\varphi}(x) = \nu_x(\partial T_w) \AND h_{\psi}(x) = \nu_x(\partial T_z).
$$
By \eqref{eq:nu}, 
$$
\begin{aligned} 
h_{\varphi}(x) &= F(x,w) \nu_w(\partial T_w)\,,
\quad x \in \{w\} \cup (T\setminus T_w)\\
1-h_{\varphi}(x) &= F(x,w) (\partial T \setminus \partial T_w)\,, 
\quad x \in T_w\,,
\end{aligned}
$$
We set $F_w(x)=F(x,w)$ and write $h_{\varphi}(x) - h_{\varphi}(x^-) =
\bigl(1-h_{\varphi}(x^-)\bigr) - \bigl(1-h_{\varphi}(x)\bigr)$ when this is convenient, 
and analogously for $h_{\psi}\,$. Then we get
$$
\begin{aligned}
&\Dir_T(h_{\varphi}\,,h_{\psi})\\ 
&= \!\!\!\sum_{[x,x^-] \in E(T) \setminus E(T_w)}\!\!\!
a(x,x^-) \bigl(F(x,w) - F(x^-,w)\bigr)\nu_w(\partial T_w)\,
\bigl(F(x,w) - F(x^-,w)\bigr)\nu_w(\partial T_z)\\
&\quad- \!\!\!\!\!\!\sum_{[x,x^-] \in E(T_w) \setminus E(T_z)} \!\!\!\!\!\!
a(x,x^-) \bigl(F(x,w) - F(x^-,w)\bigr)\nu_w(\partial T \setminus \partial T_w)\,
\bigl(F(x,w) - F(x^-,w)\bigr)\nu_w(\partial T_z)\\
&\quad+ \!\!\!\!\!\!\sum_{[x,x^-] \in E(T_z)} \!\!\!\!\!\!
a(x,x^-) \bigl(F(x,w) - F(x^-,w)\bigr)\nu_w(\partial T \setminus \partial T_w)\,
\bigl(F(x,w) - F(x^-,w)\bigr)\nu_w(\partial T \setminus \partial T_z)\\
&= \Dir_T(F_w\,,F_z)\nu_w(\partial T_w)\nu_w(\partial T_z)
- \Dir_{T_w}(F_w\,,F_z)\nu_w(\partial T_z)
+ \Dir_{T_z}(F_w\,,F_z)\nu_w(\partial T \setminus \partial T_w)\,,
\end{aligned}
$$
where of course $\Dir_{T_w}$ is the Dirichlet form of the random walk 
on the branch $T_w\,$, as discussed above, and analogously for $\Dir_{T_z}\,$.
Now $F_w = G_w/G(w,w)$ by \eqref{eq:FG}, whence \eqref{eq:Green} yields
\begin{equation}\label{eq:DirT}
\Dir_T(F_w,F_z) = \frac{\Dir_T(G_w,F_z)}{G(w,w)} = \frac{\msf(w)F(w,z)}{G(w,w)}\,.
\end{equation}
Recall that for the random walk on $T_w\,$, we have $F_{T_w}(x,w) = F(x,w),$ 
whence $G_{T_w}(x,w) = F(x,w)G_{T_w}(w,w) = F(x,w)$ for every $x \in T_w\,$.
Also, $\msf_{T_w}(w) = \msf(w)-a(w,w^-) = \msf(w)\bigl(1-p(w,w^-)\bigr)$.
We apply \eqref{eq:DirT} to that random walk and get
$$
\Dir_{T_w}(F_w,F_z) = \frac{\msf(w)(1-p(w,w^-))F(w,z)}{G_{T_w}(w,w)}\,.
$$
We now apply \eqref{eq:GU} and \eqref{eq:UF}, recalling in addition
that $p_{T_w}(w,x) = \frac{p(w,x)}{1-p(w,w^-)}$
for $x \in T_w\,$:
$$
\begin{aligned}
\frac{1-p(w,w^-)}{G_{T_w}(w,w)} &= 1-p(w,w^-) - \bigl(1-p(w,w^-)\bigr)U_{T_w}(w,w)\\ 
&= 1-p(w,w^-) - \sum_{x:x^-=w} p(w,x)F(x,w)\\
&= 1- p(w,w^-) - \bigl( U(w,w) - p(w,w^-)F(w^-,w)\bigr)\\
&= \frac{1}{G(w,w)} - p(w,w^-)\bigl(1-F(w^-,w)\bigr) 
= \frac{\nu_w(\partial T_w)}{G(w,w)} \qquad \text{by Lemma \ref{lem:nu}.}
\end{aligned}
$$
We have obtained 
$$
\Dir_{T_w}(F_w,F_z) = \frac{\msf(w)F(w,z)}{G(w,w)}\,\nu_w(\partial T_w).
$$
In the same way, exchanging roles between $T_z$ and $T_w\,$ and using reversibility
\eqref{eq:reversible},
$$
\Dir_{T_z}(F_w,F_z) = \frac{\msf(z)F(z,w)}{G(z,z)}\,\nu_z(\partial T_z)
= \frac{\msf(w)F(w,z)}{G(w,w)}\,\nu_z(\partial T_z)
= \frac{\msf(w)}{G(w,w)}\,\nu_w(\partial T_z)
$$
Putting things together, we get that
$$
\Dir_T(h_{\varphi}\,,h_{\psi})= \Dir_{T_z}(F_w\,,F_z)\nu_w(\partial T \setminus \partial T_w)
= \frac{\msf(w)}{G(w,w)}\,\nu_w(\partial T_z)
\nu_w(\partial T \setminus \partial T_w),
$$
as proposed.

\smallskip

\emph{Case 2. } $T_z \cap T_w = \emptyset\,.$

In view of Lemma \ref{lem:invariant}, both sides of equation 
\eqref{eq:doob-naim} are independent of the root $o$. Thus we may
declare our root to be one of the neighbours of $w$ that is not
on $\pi(w,z)$. Also, let $\bar w$ be the neighbour of $w$ on
$\pi(z,w)$. Then, with our chosen new root, 
the complement of the ``old'' $T_w$ is $T_{\bar w}$, which contains
$T_z$ (The latter remains the same with respect ``new'').

Thus, we can apply the result of case 1 to $T_{\bar w}$ and $T_z$.
This means that we have to replace the functions $\varphi$ and $h_{\varphi}$ with 
$1-\varphi$ and
$1-h_{\varphi}$, respectively, which just means that we change the sign on both sides
of \eqref{eq:doob-naim}. We are re-conducted to Case 1 without further
computations.

\smallskip

We deduce from what we have done so far, and from linearity of the Poisson
transform as well of bilinearity of the forms on both sides of equation 
\eqref{eq:doob-naim}, that this equation holds for linear combinations of
indicator functions of sets $\partial T_w\,$. Those indicator functions are 
dense in the space $C(\partial T)$ with respect to the $\max$-norm.
Thus, \eqref{eq:doob-naim} holds for all continuous functions on $\partial T$.
The extension to all of $\Dcal(\partial T,P)$ is by standard approximation.
\end{proof}

\end{document}